\documentclass[10pt,twoside]{siamltex}
\usepackage{amsfonts,epsfig}

\newtheorem{thm}{Theorem}[section] 
\newtheorem{algo}[thm]{Algorithm}

\newtheorem{prop}[thm]{Proposition}

\newtheorem{rem}[thm]{Remark}

\newcommand\freegen[1]{{\left<#1\right>}}
\newcommand\fring{{D\freegen{x_1,\dots,x_N}}}
\newcommand\set[1]{\{#1\}}
\def\sub{{\subseteq}}

\begin{document}



\bibliographystyle{plain}
\title{
General polynomials over division algebras and left eigenvalues\thanks{Received by the editors on Month x, 200x.
Accepted for publication on Month y, 200y   Handling Editor: .}}

\author{
Adam\ Chapman\thanks{Department of Mathematics,
Bar-Ilan University, Ramat-Gan 52115, Israel
(adam1chapman@yahoo.com). Ph.D student under the supervision of Prof. Uzi Vishne}}


\pagestyle{myheadings}
\markboth{A.\ Chapman}{General polynomials over division algebras and left eigenvalues}
\maketitle

\begin{abstract}
In this paper, an isomorphism between the ring of general polynomials over a division ring of degree $p$ over its center $F$ and the group ring of the free monoid with $p^2$ variables is presented.
Using this isomorphism, the characteristic polynomial of a matrix over any division algebra is defined, i.e. a general polynomial with one variable over the algebra whose roots are precisely the left eigenvalues.
Plus, it is shown how the left eigenvalues of a $4 \times 4$ matrices over any division algebra can be found by solving a general polynomial equation of degree $6$ over that algebra.
\end{abstract}


\section{Introduction}

\subsection{Polynomial rings over division algebras}

Let $F$ be a field and $D$ be a division ring over $F$ of degree $p$.
We adopt the terminology in \cite{GM}. Let $D_L[z]$ denote the usual
ring of polynomials over $D$ where the variable
$z$ commutes with every $d \in D$. When substituting a value we
consider the coefficients as though they are placed on the
left-hand side of the variable. The substitution map is not a ring
homomorphism in general. For example, for non-central $d \in D$ and the substitution $S_d: D_L[z] \rightarrow D$, if $f(z)=az$ and
$a d \neq d a$ then $S_d(f^2)=S_d(a^2 z^2)=a^2 d^2$ while
$S_d(f)^2=S_d(a z)^2=(a d)^2 \neq S_d(f^2)$.

The ring $D_G[z]$ is, by definition, the (associative) ring of
polynomials over $D$, where $z$ is assumed to commute with every
$d \in F=Z(D)$, but not with arbitrary elements of $D$. For
example, if $d \in D$ is non-central, then $dz^2$, $zdz$ and $z^2
d$ are distinct elements of this ring. There is a ring epimorphism
$D_G[z] \rightarrow D_L[z]$, defined by $z \mapsto z$, whose
kernel is the ideal generated by the commutators $[d,z]$ ($d \in
D$). Unlike the situation for $D_L[z]$, the substitution maps from
$D_G[z]$ to $D$ are all ring homomorphisms. Polynomials from
$D_G[z]$ are called ``general polynomials",
for example $z i z+j z i+ z i j+5 \in \mathbb{H}_G[z]$.

Polynomials in $D_L[z]$ or and polynomials in $D_G[z]$ which ``look like" polynomials in $D_L[z]$, i.e. the coefficients are placed on the left-hand side of the variable, are called ``left" or ``standard polynomials", for example $z^2+i z+j \in \mathbb{H}_G[z]$.

Let $\fring$ be the ring of multi-variable polynomials, where for all $i$, $x_i$ commutes with every $d \in D$ and is not assumed to commute with $x_j$ for $i \neq j$. This is the group algebra
of the free monoid $\freegen{x_1,\dots,x_N}$ over $D$. The
commutative counterpart is $D_L[x_1,\dots,x_N]$, which is the ring
of multi-variable polynomials where for all $i$, $x_i$ commutes
with every $d \in D$ and with every $x_j$ for $i \neq j$.

For further reading on what is generally known about polynomial
equations over division rings see \cite{L}.

\subsection{Left eigenvalues of matrices over division algebras}

Given a matrix $A \in M_n(D)$, a left eigenvalue of $A$ is an element $\lambda \in D$ for which there exists a nonzero vector $v \in D^{n \times 1}$ such that $A v=\lambda v$.

For the special case of $D=\mathbb{H}$\footnote{The algebra of real quaternions} and $n=2$ it was proven by Wood in \cite{Wood} that the left eigenvalues of $A$ are the roots of a standard quadratic quaternion polynomial.
In \cite{So} proved that for $n=3$, the left eigenvalues of $A$ are the roots of a general cubic quaternion polynomial.

In \cite{Macias-Virgos}, Mac\'{i}as-Virg\'{o}s and Pereira-S\'{a}ez gave another proof to Wood's result. Their proof makes use of the Study determinant.

Given a matrix $A \in M_n(\mathbb{H})$, there exist unique matrices $B,C \in M_n(\mathbb{C})$ such that $A=B+C j$.
The Study determinant of $A$ is
$\det \left[ \begin{array}{lr} B  & -\overline{C} \\ C &  \overline{B}  \end{array}
\right]$.
The Dieudonn\'{e} determinant is (in this case) the square root of the Study determinant\footnote{In \cite{Macias-Virgos} the Study determinant is defined to be what we call the Dieudonn\'{e} determinant.}.
For further information about these determinants see \cite{Aslaksen}.

\section{The isomorphism between the ring of general polynomials and the group ring of the free monoid with $[D:F]$ variables}

Let $N=p^2$, i.e. $N$ is the dimension of $D$ over its center $F$. In particular there exist invertible elements $a_1,\dots,a_{N-1} \in D$ such that $D=F+a_1 F+\dots+a_{N-1} F$.

Let $h : D_G[z] \rightarrow \fring$ be the homomorphism for which $h(d)=d$ for all $d \in D$, and $h(z)=x_1+a_1 x_2+\dots+a_{N-1} x_N$.
$D_L[x_1,\dots,x_N]$ is a quotient ring of $\fring$.
Let $g: \fring \rightarrow D_L[x_1,\dots,x_N]$ be the standard epimorphism.

In \cite[Theorem~6]{GM} it says that if $D$ is a division algebra then homomorphism $g \circ h : D_G[z] \rightarrow D_L[x_1,\dots,x_N]$ is an epimorphism. The next theorem is a result of this fact.

\begin{thm} \label{isom}
The homomorphism $h : D_G[z] \rightarrow \fring$ is an isomorphism, and therefore $D_G[z] \cong \fring$.
\end{thm}

\begin{proof}
$h$ is well-defined because $z$ commutes only with the center.

Both $D_G[z]$ and $\fring$ can be graded, $D_G[z]=G_0 \bigoplus G_1 \bigoplus \dots$ and $\fring=H_0 \bigoplus H_1 \bigoplus \dots$ such that for all $n$, $G_n$ and $H_n$ are spanned by monomials of degree $n$.

For all $n$, $h(G_n) \subseteq H_n$.
Furthermore, the basis of $G_n$ as a vector space over $F$ is $\set{b_1 z b_2 \dots b_n z b_{n+1} : b_1,\dots,b_{n+1} \in {1,a_1,\dots,a_{N-1}}}$, which means that $[G_n:F]=N^{n+1}$.
Plus, the basis of $H_n$ as a vector space over $F$ is $\set{b x_{k_1} x_{k_2} \dots x_{k_n} : b \in \set{1,a_1,\dots,a_{N-1}}, \forall_j k_j \in \set{1,\dots,N}}$, hence $[H_n:F]=N \cdot N^n=N^{n+1}=[G_n:F]$.

Consequently, it is enough to prove that $h|_{G_n} : G_n \rightarrow H_n$ is an epimorphism.
For that, it is enough to prove that for each $1 \leq k \leq N$, $x_k$ has a co-image in $G_1$.
The reduced epimorphism $g|_{H_1}$ is an isomorphism\footnote{An easy exercise}, and since $g \circ h|_{G_1}$ is an epimorphism, $h|_{G_1}$ is also an epimorphism, and that finishes the proof.
\end{proof}

Here is a suggested algorithm for finding the co-image of $x_k$ for any $1 \leq k \leq N$:

\begin{algo}
Let $p_1=z$, therefore $h(p_1)=x_1+a_1 x_2+\dots+a_{N-1} x_N$.
We shall define a sequence $\set{p_j : j=1,\dots,n} \sub G_1$ as follows:
If there exists a monomial in $h(p_j)$ whose coefficient $a$ does not commute with the coefficient of $x_k$, denoted by $c$, then we shall define $p_{j+1}=a p_j a^{-1}-p_j$, by which we shall annihilate at least one monomial (the one whose coefficient is $a$), and yet the element $x_k$ will not be annihilated, because $c x_k$ does not commute with $a$.

If $c$ commutes with all the other coefficients then we shall pick some monomial which we want to annihilate. Let $b$ denote its coefficient. Now we shall pick some $a \in D$ which does not commute with $c b^{-1}$ and define $p_{j+1}=b a p_j b^{-1} a^{-1}-p_j$.

The element $x_k$ is not annihilated in this process,
because if we assume that it does at some point, let us say it is annihilated in $h(p_{j+1})$, then $b a c b^{-1} a^{-1}-c=0$.
Therefore $c^{-1} b a c b^{-1} a^{-1}=1$, hence $c b^{-1} a^{-1}=(c^{-1} b a)^{-1}=a^{-1} b^{-1} c$ and, since $b$ commutes with $c$, $a$ commutes with $c b^{-1}$ and that creates a contradiction.

In each iteration the length of $h(p_j)$ (the number of monomials in it) decreases by at least one, and yet the element $x_k$ always remains, and since the length of $h(p_1)$ is finite, this process will end with some $p_m$ for which $h(p_m)$ is a monomial.
In this case, $h(q_m)=c x_k$ and consequently $x_k=h(c^{-1} q_m)$.
\end{algo}

\subsection{Real Quaternions}\label{example}
Let $D=\mathbb{H}=\mathbb{R}+i \mathbb{R}+j \mathbb{R}+i j \mathbb{R}$. Now $h(z)=x_1+x_2 i+x_3 j+x_4 i j$\\
$h(z-j z j^{-1})=h(z+j z j)=2 x_2 i+2 x_4 i j$\\
$h((z+j z j)-i j (z+j z j) (i j)^{-1})=2 x_2 i+2 x_4 i j-i j (2 x_2 i+2 x_4 i j) (i j)^{-1}=4 x_2 i$\\ therefore $h^{-1}(x_2)=-\frac{1}{4} i ((z+j z j)+i j (z+j z j) i j))=-\frac{1}{4} (i z+i j z j-j z i j+z i)$.

Similarly, $h^{-1}(x_1)=\frac{1}{4} (z-i z i-j z j-i j z i j)$, $h^{-1}(x_3)=-\frac{1}{4} (j z-i j z i+i z i j+z j)$ and $h^{-1}(x_4)=-\frac{1}{4} (i j z-i z j+j z i+z i j)$.
Consequently, $\overline{z}=\overline{h^{-1}(x_1+x_2 i+x_3 j+x_4 i j)}=h^{-1}(x_1-x_2 i-x_3 j-x_4 i j)
=-\frac{1}{2} (z+i z i+j z j+i j z i j)$.

\section{The characteristic polynomial}\label{char}
Let $D,F,p,N$ be the same as they were in the previous section.

There is an injection of $D$ in $M_p(K)$ where $K$ is a maximal subfield of $D$.
(In particular, $[K:F]=p$.)
More generally, there is an injection of $M_k(D)$ in $M_{k p}(K)$ for any $k \in \mathbb{N}$.
Let $\widehat{A}$ denote the image of $A$ in $M_{k p}(K)$ for any $A \in M_k(D)$.

The determinant of $\widehat{A}$ is equal to the Dieudonn\'{e} determinant of $A$ to the power of $p$.
\footnote{The reduced norm of $A$ is defined to be the determinant of $\widehat{A}$}

Therefore $\lambda \in D$ is a left eigenvalue of $A$ if and only if $\det(\widehat{A-\lambda I})=0$.
Considering $D$ as an $F$-vector space $D=F+F a_1+\dots+F a_{N-1}$, we can write $\lambda=x_1+x_2 a_1+\dots+x_N a_{N-1}$ for some $x_1,\dots,x_N \in F$. Then $\det(\widehat{A-\lambda I}) \in F[x_1,\dots,x_N]$.
It can also be considered as a polynomial in $\fring$.
Now, there is an isomorphism $h : D_G[z] \rightarrow \fring$, and so $h^{-1} (\det(\widehat{A-\lambda I})) \in D_G[z]$.

Defining $p_A(z)=h^{-1} (\det(\widehat{A-\lambda I}))$ to be the characteristic polynomial of $A$, we get that the left eigenvalues of $A$ are precisely the roots of $p_A(z)$.

The degree of the characteristic polynomial of $A$ is therefore $k p$.

\begin{rem}
If one proves that the Dieudonn\'{e} determinant of $A-\lambda I$ is the absolute value of some polynomial $q(x_1,\dots,x_N) \in D_L[x_1,\dots,x_N]$ then we will be able to define the characteristic polynomial to be $h^{-1}(q(x_1,\dots,x_N))$ and obtain a characteristic polynomial of degree $k$.
\end{rem}

\section{The left eigenvalues of a $4 \times 4$ quaternion matrix}

Let $Q$ be a quaternion division $F$-algebra.
Calculating the roots of the characteristic polynomial as defined in Section \ref{char} is not always the best way to obtain the left eigenvalues of a given matrix.

The reductions Wood did in \cite{Wood} and So did in \cite{So} suggest that in order to obtain the left eigenvalues of a $2 \times 2$ or $3 \times 3$ matrix one can calculate the roots of a polynomial of degree $2$ or $3$ respectively, instead of calculating the roots of the characteristic polynomial whose degree is $p$ times greater.

In the next proposition we show how (under a certain condition) the eigenvalues of a $4 \times 4$ quaternion matrix can be obtained by calculating the roots of three polynomials of degree $2$ and one of degree $6$.

\begin{prop}
If $M=
\left[ \begin{array}{lr} A  & B \\ C &  D  \end{array}
\right]$ where $A,B,C,D \in M_2(Q)$ and $C$ is invertible then $\lambda$ is a left eigenvalue of $M$ if and only if it is either $e(\lambda)=f(\lambda) g(\lambda)=0$ or $e(\lambda) \neq 0$ and $e(\lambda)\overline{e(\lambda)} h(\lambda)-g(\lambda) \overline{e(\lambda)} f(\lambda)=0$ where $C (A-\lambda I) C^{-1} (D-\lambda I)-C B=
\left[ \begin{array}{lr} e(\lambda)  & f(\lambda) \\ g(\lambda) &  h(\lambda)  \end{array}
\right]$
\end{prop}

\begin{proof}
Let $M$ be a $4 \times 4$ quaternion matrix.
Therefore $M=
\left[ \begin{array}{lr} A  & B \\ C &  D  \end{array}
\right]$ where $A,B,C,D$ are $2 \times 2$ quaternion matrices.

An element $\lambda$ is a left eigenvalue if $\det(M-\lambda I)=0$.
Assuming that $\det(C) \neq 0$, we have $\det(M-\lambda I)=\det(C (A-\lambda I) C^{-1} (D-\lambda I)-C B)$ (This is an easy result of the Schur complements identity for complex matrices extended to quaternion matrices in ).

The matrix $C (A-\lambda I) C^{-1} (D-\lambda I)-C B$ is equal to $
\left[ \begin{array}{lr} e(\lambda)  & f(\lambda) \\ g(\lambda) &  h(\lambda)  \end{array}
\right]$
for some quadratic polynomials $e,f,g,h$.

Now, if $e(\lambda) \neq 0$ then $\det
\left[ \begin{array}{lr} e(\lambda)  & f(\lambda) \\ g(\lambda) &  h(\lambda)  \end{array}
\right]=0$ if and only if \\
$h(\lambda)-g(\lambda) e(\lambda)^{-1} f(\lambda) = 0$.

This happens if and only if $e(\lambda)\overline{e(\lambda)} h(\lambda)-g(\lambda) \overline{e(\lambda)} f(\lambda)=0$.

\end{proof}

As we saw in Subsection \ref{example}, $\overline{e(\lambda)}$ is also a quadratic polynomial,
which means that $e(\lambda)\overline{e(\lambda)} h(\lambda)-g(\lambda) \overline{e(\lambda)} f(\lambda)$ is a polynomial of degree $6$, while the characteristic polynomial of $M$ as defined in Section \ref{char} is of degree $8$.


\end{document}